\documentclass[12pt]{article}
\usepackage{verbatim}
\usepackage{amssymb}
\usepackage{amsthm}

\usepackage{amsmath}
\usepackage{mathrsfs}
\usepackage[margin=1in]{geometry}
\usepackage[usenames,dvipsnames] {color}
\numberwithin{equation}{section}

\title{Asymptotic M\"untz-Sz\'asz Theorems}

\newcommand{\mc}{M\raise.45ex\hbox{c}Carthy}
\author{Jim Agler
and
John E. M\raise.45ex\hbox{c}Carthy
\thanks{Partially supported by National Science Foundation Grant  
DMS 2054199}
}

\def\d{\mathbb{D}}
\def\t{\mathbb{T}}
\def\h{\mathcal{H}}

\def\a{\mathcal{A}}

\def\be{\begin{equation}}
\def\ee{\end{equation}}

\def\m{\mathcal{M}}

\def\n{\mathcal{N}}

\def\c{\mathbb{C}}
\def\d{\mathbb{D}}
\def\t{\mathbb{T}}
\def\s{\mathbb{S}}

\def\r{\mathbb{R}}

\def\h{\mathcal{H}}

\def\ltwo{\ell^2}
\def\set#1#2{\{ #1 \, | \, #2\}}

\def\htwo{{\rm H}^2}
\def\ltwo{{\rm L}^2([0,1])}
\def\ltworho{{\rm L}^2 ([\rho,1])}
\def\ltwoinf{{\rm L}^2(0,\infty)}

\def\a{\mathcal{A}}

\def\C{\mathscr{C}}

\def\lap#1{\mathcal{L} \{#1\}}

\def\ip#1#2{\langle\, #1\, ,#2\, \rangle}

\def\b{\mathcal{B}}

\newcommand{\twopartdef}[4]
{
	\left\{
		\begin{array}{ll}
			#1 & \mbox{if } #2 \\
			#3 & \mbox{if } #4
		\end{array}
	\right.
}

\DeclareMathOperator{\ran}{ran}

\DeclareMathOperator{\re}{Re}
\DeclareMathOperator{\dist}{dist}

\renewcommand\={\ =\ }
\newcommand\D{{\mathbb D}}

\newcommand\bt{\begin{thm}}
\newcommand\et{\end{thm}}
\newcommand\beq{\begin{eqnarray*}}
\newcommand\eeq{\end{eqnarray*}}
\newcommand\bp{\begin{proof}}
\newcommand\ep{\end{proof}}
\newcommand\la{\langle}
\newcommand\ra{\rangle}
\newcommand\ej{Erd\'elyi-Johnson\ }

\newcommand\A{{\mathcal A}}
\renewcommand\i{\infty}
\renewcommand\Re{{\mathrm Re\,}}

\renewcommand\wp{w^\prime}

\begin{document}

\bibliographystyle{plain}
\theoremstyle{definition}
\newtheorem{defin}[equation]{Definition}
\newtheorem{lem}[equation]{Lemma}
\newtheorem{lemma}[equation]{Lemma}
\newtheorem{prop}[equation]{Proposition}
\newtheorem{thm}[equation]{Theorem}
\newtheorem{theorem}[equation]{Theorem}
\newtheorem{claim}[equation]{Claim}
\newtheorem{ques}[equation]{Question}
\newtheorem{fact}[equation]{Fact}
\newtheorem{axiom}[equation]{Technical Axiom}
\newtheorem{newaxiom}[equation]{New Technical Axiom}
\newtheorem{cor}[equation]{Corollary}
\newtheorem{exam}[equation]{Example}
\newtheorem{rem}[equation]{Remark}
\newtheorem{problem}[equation]{Open Problem}
\maketitle
\begin{abstract}
We define a monomial space to be a subspace of $\ltwo$ that can be approximated by spaces that are spanned by monomial functions. 
We describe the structure of monomial spaces.
\end{abstract}

\section{Introduction}
What sorts of subspaces in $\ltwo$ can be limits of spans of monomials?
Specifically, let
\[
\mathbb{S}=\set{s\in \c}{\re s > -\frac{1}{2}},
\]
so that $s\in \s$ if and only if $x^s \in \ltwo$. For $S$ a finite subset of $\s$ we let $\m(S)$ denote the span in $\ltwo$ of the monomials whose exponents lie in $S$, i.e.,
\[
\m(S)=\set{\sum_{s\in S}a(s)x^s}{a:S\to \c}.
\]
We refer to sets in $\ltwo$ that have the form $\m(S)$ for some finite subset $S$ of $\s$ as \emph{finite monomial spaces}.
We are interested in what the limits of such spaces are.
\begin{defin}
\label{defin1}
If $\m$ is a subspace of a Hilbert space $\h$  and $\{\m_n\}$ is a sequence of closed subspaces, we say that \emph{$\{\m_n\}$ tends to $\m$} and write
\[
\m_n \to \m\ \text{ as }\ n\to \infty
\]
if
\[
\m = \set{f \in \h}{\lim_{n \to \infty} \dist(f,\m_n)=0}.
\]
\end{defin}
There are alternative ways to frame this definition; see Proposition \ref{proppf1}.

\begin{defin}
We say that a subspace $\m$ of $\ltwo$ is a \emph{monomial space} if there  exists a sequence $\{\m_n\}$ of finite monomial spaces such that $\m_n \to \m$.
\end{defin}
The goal of this paper is to study monomial spaces, which have a rich structure and are intimately related to the M\"untz-Szasz theorem and generalizations thereof.

\subsection{Monotone Monomial Spaces}
In this subsection we recall several classical results that can be interpreted as facts about monomial spaces. In each example we consider there is a limit $\m_n \to \m$ of finite monomial spaces that is monotone, i.e.,
\[
\m_i \subseteq \m_j\ \ \text{ whenever }\ \ i\le j.
\]
For a detailed account of the results in this section, see \cite{alm07} and \cite{boer95}.
\begin{exam}{\bf The Weierstrass Approximation Theorem}
Let $S_n = \{0, 1,2,\ldots,n\}$. The Weierstrass Theorem, which implies that the polynomials are dense in $\ltwo$ implies that
\[
\m(S_n) \to \ltwo.
\]
In particular, $\ltwo$ is a monomial space.
\end{exam}
\begin{exam}\label{int.exam.10}{\bf Classical M\"untz-Sz\'asz Theorem}
This was proved in \cite{mu14,sza16}.
Fix a strictly increasing sequence of nonnegative integers $s_0,s_1,s_2,\ldots,$ and let
\[
S_n = \{s_0,s_1,\ldots,s_n\}.
\]
Then there exists a space $\m$ such that
\[
S(\m_n) \to \m.
\]
Furthermore,
\[
\m = \ltwo\ \ \  \text{ if and only if }\ \ \ \sum_{k=1}^\infty \frac{1}{s_k} =\infty.
\]
\end{exam}
\begin{exam}{\bf 
Sz\'asz's Theorem (real case), (a.k.a Full M\"untz-Sz\'asz Theorem in ${\bf \ltwo}$)}.
This was proved in \cite{sza53}.
Fix a sequence of distinct real numbers $s_0,s_1,s_2,\ldots,$ in $\s$ and let
\[
S_n = \{s_0,s_1,\ldots,s_n\}.
\]
Then there exists a space $\m$ such that
\[
\m(S_n) \to \m.
\]
Furthermore,
\[
\m = \ltwo\ \ \  \text{ if and only if }\ \ \ \sum_{k=0}^\infty \frac{2 s_k +1}{(2s_k+1)^2 +1} =\infty.
\]
\end{exam}

\begin{exam}\label{int.exam.20}{\bf Sz\'asz's Theorem (complex case)}
Any of the proofs known to the authors of the previous example, including Sz\'asz's original proof, can be adapted to show that if $s_0,s_1,s_2,\ldots,$ is a sequence of distinct points in $\s$ and
\[
S_n = \{s_0,s_1,\ldots,s_n\},
\]
then there exists a space $\m$ such that
\[
\m(S_n) \to \m,
\]
and where,
\[
\m = \ltwo\ \ \  \text{ if and only if }\ \ \ \sum_{k=1}^\infty \frac{2 \re s_k +1}{|s_k +1 |^2} =\infty.
\]
\end{exam}
What happens in the above examples when $\m \not= \ltwo$?
\begin{exam}
{\bf The Clarkson-Erd\H{o}s Theorem.}
With the setup of Example \ref{int.exam.10}, assume that $\m \not= \ltwo$. Then Clarkson and Erd\H{o}s proved
in \cite{ce43} that  the elements of $\m$ extend to be analytic on $\d$! Furthermore, if $f \in \m$, then $f$ has a power series representation of the form
\[
f(z) = \sum_{k=0}^\infty a_k z^{s_k},\qquad z\in \d.
\]
This result was generalized to arbitrary real powers in $(-\frac{1}{2}, \infty)$ by Erd\'elyi and Johnson \cite{ej01}, 
who showed that if $\m \not= \ltwo$, then every $f$ in $\m$ is analytic in ${\mathbb D} \setminus (-1,0]$
\end{exam}
In honor of this remarkable theorem we introduce the following definition.
\begin{defin}
We say that \emph{$\m$ is a Clarkson-Erd\H{o}s space} if there exist a sequence $\{s_0,s_1,\ldots\}$ in $\s$
of distinct points
such that
\[
\m(\{s_0,s_1,\ldots,s_n\})\to \m
\]
where $\m \not= \ltwo$.
\end{defin}
We want to allow for multiplicities.
If an entry $s$ is repeated in a sequence, this corresponds to multiplicity in the following way.
The first occurrence of $s$ in $S_n$ gives the function $x^s$ in $\m_n$. The second occurrence
gives $\frac{\partial}{\partial s} x^s =  x^s \log x$. If $s$ occurs $k$ times, then $\m_n$ contains
the functions $ x^s,  x^s \log x, \dots ,  x^s (\log x)^{k-1}$.
This leads to the following generalization of a  Clarkson-Erd\H{o}s space.
\begin{defin}
We say that \emph{$\m$ is an  \ej space} if there exist a sequence $\{s_0,s_1,\ldots\}$ in $\s$, with multiplicities allowed,
such that
\[
\m(\{s_0,s_1,\ldots,s_n\})\to \m
\]
where $\m \not= \ltwo$.
\end{defin}

\subsection{A Non-monotone Monomial Space}
In the study of monomial spaces it is natural to consider the class of \emph{monomial operators}, i.e., the class of bounded operators $T$ acting on $\ltwo$ that take monomials to monomials, i.e.,
\be\label{int.10}
\forall_{s \in \s}\ \ \exists_{\tau \in \s}\ \ \exists_{c \in \c}\ \  Tx^s = cx^\tau.
\ee
In \cite{amHI} the authors studied the special case  wherein it is assumed that there exists a fixed number $m$ such that in \eqref{int.10} $\tau$ can be chosen to equal $s+m$ for all $s$. We call these {\em flat monomial operators.} 
In the course of proving that flat monomial operators leave ${\rm L}^2 ([a,1])$ invariant for each $a \in [0,1]$, the authors discovered the following example.
\begin{exam}
Fix $\rho \in [0,1]$ and choose an increasing sequence of integers $N_1,N_2,\ldots$ such that
\[
\lim_{n\to \infty} \frac{n}{n+N_n} = \sqrt\rho.
\]
If
\[
S_n = \{n+1,n+2,\ldots,n+N_n\},
\]
then
\[
\m(S_n) \to {\rm L}^2 ([\rho,1]).
\]
In particular, $\ltworho$ is a monomial space for each $\rho \in [0,1]$, where here and afterwards
we identify $\ltworho$ with the subspace of $\ltwo$ consisting of functions that vanish a.e. on $[0,\rho]$.
\end{exam}

\subsection{Characterization of monomial spaces}

The Hardy operator $H : \ltwo \to \ltwo$ is defined by
\be
\label{eqin2}
H f(x) \ =\ \frac{1}{x} \int_0^x f(t) dt .
\ee
This was introduced by Hardy in \cite{har20}, where he proved it was bounded.
As $H x^s = \frac{1}{s+1} x^s$ for all $s \in \s$, the Hardy operator leaves invariant every monomial space.
The converse is true.
\begin{thm}
\label{thmbeur}
A closed subspace of $\ltwo$ is a monomial space if and only if it is invariant for $H$.
\end{thm}
A proof of \ref{thmbeur} that uses real analysis techniques is given in \cite{amBT}.
\begin{cor}
\label{cormo}
A bounded operator $T$ on $\ltwo$ is a monomial operator if and only if for every $\m \in {\rm Lat}(H)$,
the space $T \m$ is in ${\rm Lat} (H)$.
\end{cor}

\subsection{A Decomposition Theorem}
\begin{defin}
We say a space $\m$ in $\ltwo$ is a \emph{singular space} if $\m$ is a monomial space that does not contain any Clarkson-Erd\H{o}s space.
\end{defin}
\begin{thm}
\label{thmdecomp}
Every monomial space $\m$ has a unique decomposition,
\[
\m = \overline{\m_{0}+\m_1},
\]
where $\m_0$ is an \ej space and $\m_1$ is a singular space.
\end{thm}

\subsection{Atomic Spaces}

Unitary monomial operators can be characterized using a theorem of Bourdon and Narayan \cite{bona10}. Their description is equivalent to the following reformulation, from \cite{amMO}.
\begin{thm}
\label{thmun}
The operator 
\[
T : x^s \mapsto c(s) x^{\tau(s)} 
\]
is a unitary map from $\ltwo$ to $\ltwo$ if and only if $\tau$ is a holomorphic automorphism of $\s$ and
$c$ is given by
\[
c(s) \= c_0 \frac{1 + \overline{\tau (0)} + \tau(s)}{1 + s} ,
\]
where $c_0$ is a constant
satisfying
\[
| c_0 | \= \frac{1}{\sqrt{1 + 2 \Re \tau(0)}} .
\]
\end{thm}

\begin{defin}
We say a space $\m$ in $\ltwo$ is \emph{atomic} if there exist $\rho \in (0,1)$ and a unitary monomial operator $T$ such that $\m = T\ltworho$.
\end{defin}
We can describe all atomic spaces in the following way.
\begin{thm}
\label{thmbasis}
The functions
\be
\label{eqin1}
e_n(x) \=\sum_{k=0}^n \binom{n}{k}\frac{(\ln x)^k}{k!},\qquad n\ge 0.
\ee
form an orthonormal basis for $\ltwo$. 
 Furthermore, the operator $J$ defined on $\ltwo$ by requiring
\[
J(e_n)=\twopartdef{e_n}{n \text{ is even}}{-e_n}{n \text{ is odd}}
\]
is a unitary monomial operator, corresponding to the choice
\be
\label{eqin21}
\tau(s)  \= \frac{-s}{1+2s},\qquad c(s) \= \frac{1}{1+2s} 
\ee
in Theorem \ref{thmun}.
\end{thm}
The polynomials
\[
p_n(t) \= \sum_{k=0}^n \binom{n}{k}\frac{( t)^k}{k!},\qquad n\ge 0,
\]
are the Laguerre polynomials. They are the orthogonal polynomials on $[0,\infty)$ for the measure $e^{-t}dt$.
Under the change of variables $x = e^{-t}$ they become the functions $e_n$ in \eqref{eqin1}.
Their connection to the Hardy operator was shwon in \cite{kmp06}.

If $c$ is real, the multiplication operator $M_{x^{ic}}$ is also a unitary monomial operator.
These two operators can be used to build the general atomic space.
\begin{enumerate}
\item For any $w >0$, define
\[
\a_{1,w} \ :=\  L^2([e^{-2w},1]).
\]
\item
Define
\[
\a_{-1,w} \ :=\ J\, \a_{1,w}
\]
\item For any  $\tau \in {\mathbb T} \setminus \{1 \}$, define
\[
\a_{\tau,w} \ :=\ M_{x^{ic}}\, \a_{-1,\wp},\qquad {\rm where\ }2 ic=\frac{\tau +1}{\tau -1}, \ \wp = (1 + 4c^2) w.
\]
\end{enumerate}
(The reason for the strange scaling is to simplify the formulas in Section \ref{sechar}).
Atomic spaces are all of the form $\a_{\tau,w}$.
\begin{thm}
\label{thmsing1}
Every atomic space is equal to $\a_{\tau,w}$ for exactly one pair $(\tau,w) \in {\mathbb T} \times (0,\infty)$.
\end{thm}

\subsection{The Structure of Singular Spaces}

 Say $\m$ is \emph{finitely atomic} if $\m$ is a finite sum of atomic spaces. 
 If 
\[
\mu = \sum_{k=1}^n \ w_k \ \delta_{\tau_k}
\]
is a finitely atomic measure on $\t$, with distinct atoms $\tau_k$, we define a finitely atomic space in $\ltwo$ by the formula
\[
\m(\mu)=\sum_{k=1}^n \a_{\tau_k,w_k}
\]
 \begin{thm}
 \label{thmsing2}

 The assignment $\mu \mapsto \m(\mu)$ extends by weak-* sequential continuity to 
a map from the positive singular Borel measures on $\t$ into closed subspaces of $\ltwo$. When extended,
     \[
     \mu_n \to \mu \text{ weak-* } \implies \m(\mu_n) \to \m(\mu).
     \]
\end{thm}

\subsection{The main idea}
\label{subsecpt}

The Hardy operator is unitarily equivalent to $1- S^*$, where $S$ is the unilateral shift, via a unitary
$U : \ltwo \to H^2$ that we call the Sarason transform, described in Section \ref{secsa}.
It follows that the invariant subspaces of $H$ can be described by Beurling's theorem \cite{beu} in terms of model spaces, the invariant subspaces for the backward shift described for example in \cite{gmr16}.
However, all the theorems above have been stated in terms that are {\em intrinsic} to $\ltwo$.
We believe that finding proofs that are also intrinsic to $\ltwo$ will illuminate this space with a new light.
So far, the authors have only succeeded in doing this for some of these results. 
\begin{problem}
Find real analysis proofs to Theorems \ref{thmdecomp} and \ref{thmsing2}.
\end{problem}

\section{The Sarason Transform}
\label{secsa}
\subsection{The Definition}
We let $k_\alpha$ denote the Szeg\"o kernel function for $\htwo$, the classical Hardy space of square integrable functions on $\d$, i.e.,
\[
k_\alpha (z) = \frac{1}{1-\bar\alpha z},\qquad z \in \d.
\]

As the monomials are linearly independant in $\ltwo$, there is a well defined map $L$ defined on polynomials in $\ltwo$ into $\htwo$ defined by the formula
\[
L(\sum_{n=0}^N a_nx^n) = \sum_{n=0}^N a_n \frac{1}{n+1}k_{\frac{n}{n+1}}
\]
Noting that
\[
\ip{L(x^i)}{L(x^j)}_{\htwo} = \ip{x^i}{x^j}_{\ltwo}
\]
for all nonnegative integers $i$ and $j$, it follows that
\[
\ip{L(p)}{L(q)}_{\htwo} = \ip{p}{q}_{\ltwo}
\]
for all polynomials $p$ and $q$, i.e., $L$ is isometric. Hence, as the polynomials are dense in $\ltwo$, $L$ has a unique extension to an isometry $U$ defined on all of $\ltwo$. Finally, noting that
\[
\set{\frac{n}{n+1}}{n \text{ is a nonegative integer}}
\]
is a set of uniqueness for $\htwo$, it follows that the range of $L$ is dense in $\htwo$, which implies that $U$ is a unitary transformation from $\ltwo$ onto $\htwo$.
\begin{defin}
\label{defst}
We let $U$ denote the unique unitary transformation from $\ltwo$ onto $\htwo$ that satisfies
\[
U(x^n) = \frac{1}{n+1} k_{\frac{n}{n+1}}
\]
for all nonnegative integers $n$.
\end{defin}
We call $U$ the Sarason transform, as it is similar to the transform from $L^2([0,\infty))$ onto $\htwo$ used in 
\cite{sar65c}.
\subsection{Moments in $\ltwo$ and Interpolation in $\htwo$}
As the monomials are dense in $\ltwo$, a function $f\in \ltwo$ is uniquely determined by its moment sequence
\[
\int_0^1 x^n\ f(x)dx,\qquad n=0,1,\ldots.
\]
Similarly, as the sequence $\set{1-\tfrac{1}{n+1}}{n=0,1,\ldots}$ is a set of uniqueness for $\htwo$, a function $h \in \htwo$ is the unique solution $g$ in $\htwo$ to the interpolation problem
\[
g(\frac{n}{n+1}) = h(\frac{n}{n+1}),\qquad n=0,1, \ldots
\]
The following proposition follows immediately from Definition \ref{defst}.
\begin{prop}
Fix a sequence of complex numbers $w_0,w_1,w_2,\ldots$. If $f$ in $\ltwo$ solves the moment problem
\[
\int_0^1 x^n\ f(x)\ dx=w_n,\qquad n=0,1,\ldots,
\]
then $Uf\in \htwo$ and solves the interpolation problem
\[
Uf(\frac{n}{n+1})=(n+1)w_n,\qquad n=0,1,\ldots\ .
\]
If $h \in \htwo$ solves the interpolation problem
\[
h(\frac{n}{n+1}) =w_n,\qquad n=0,1,\ldots ,
\]
then $U^*h \in \ltwo$ and solves the moment problem
\[
\int_0^1 x^n\ U^*h(x)\ dx=\frac{1}{n+1} w_n,\qquad n=0,1,\ldots\ .
\]
\end{prop}
The correspondence between moments and interpolation described in the preceding proposition allows us to easily calculate the Sarason transform of many common functions. We illustrate this with the following two lemmas.
\begin{lem}\label{sa.lem.10}
If $\alpha \in \d$, then
\be\label{40}
U^* (k_\alpha) (x) = \frac{1}{1-\bar\alpha} x^{\tfrac{\bar\alpha}{1-\bar\alpha}},\qquad x\in[0,1].
\ee
If $Re\ \beta > -\frac12$, then
\[
U (x^{\beta} ) \= \frac{1}{\beta+1} k_{\frac{\bar\beta}{\bar\beta+1}}.
\]
\end{lem}
\begin{proof}
We note that the two assertions of the lemma are equivalent. Therefore it suffices to prove \eqref{40}. Since the left and right hand sides of \eqref{40} are in $\ltwo$, to show \eqref{40} it suffices to show that for each $n \ge 0$
\[
\ip{x^n}{U^*k_\alpha}_{\ltwo} =
\ip{x^n}{\frac{1}{1-\bar\alpha} x^{\tfrac{\bar\alpha}{1-\bar\alpha}}}_{\ltwo} .
\]
But
\begin{align*}
\ip{x^n}{\frac{1}{1-\bar\alpha} x^{\tfrac{\bar\alpha}{1-\bar\alpha}}}_{\ltwo}
&=\frac{1}{1-\alpha}\ip{x^n}{ x^{\tfrac{\bar\alpha}{1-\bar\alpha}}}_{\ltwo}\\
&=\frac{1}{1-\alpha}\int_0^1\ x^n \ \overline{x^{\tfrac{\bar\alpha}{1-\bar\alpha}}}\ dx\\
&=\frac{1}{1-\alpha} \  \frac{1}{ n+ \tfrac{\alpha}{1-\alpha}+1}\\
&=\frac{1}{n+1} \ \frac{1}{1-\frac{n}{n+1}\alpha}\\
&=\frac{1}{n+1} k_{\frac{n}{n+1}}(\alpha)\\
&= (Ux^n)(\alpha)\\
&=\ip{Ux^n}{k_\alpha}_{\htwo}\\
&=\ip{x^n}{U^*k_\alpha}_{\ltwo}
\end{align*}
\end{proof}
For $S$ a measurable set in $[0,1]$ let $\chi_S$ denote the characteristic function of $S$.
\begin{lem}
\label{lemchi}
If $s\in [0,1]$, then
\be\label{50}
U\chi_{[0,s]} (z) = \sqrt s\  e^{ \tfrac12 \ln s \  \frac{1+z}{1-z}}
\ee
\end{lem}
\begin{proof}
We first observe that
\begin{align*}
\frac{s^{n+1}}{n+1}&= \int_0^s x^n\ dx\\
&=\ip{\chi_{[0,s]}}{x^n}_{\ltwo}\\
&=\ip{U\chi_{[0,s]}}{U x^n}_{\htwo}\\
&=\ip{U\chi_{[0,s]}}{\frac{1}{n+1}k_{\frac{n}{n+1}}}_{\htwo}\\
&=\frac{1}{n+1}U\chi_{[0,s]}(\frac{n}{n+1}),
\end{align*}
so that
\[
U\chi_{[0,s]}(\frac{n}{n+1})=s^{n+1}
\]
for all $n\ge 0$. On the other hand if for $w>0$ we let $E_w$ denote the singular inner function defined by
\[
E_w(z) = e^{-w\frac{1+z}{1-z}}=e^w e^{-\frac{2w}{1-z}},
\]
we have that
\begin{align*}
E_w(\frac{n}{n+1})&=e^w e^{-2w(n+1)}=e^w (e^{-2w})^{n+1}
\end{align*}
for all $n\ge 0$. Hence if we choose $w=-\tfrac12 \ln s$,
\[
U\chi_{[0,s]}(\frac{n}{n+1}) = e^{-w} E_w (\frac{n}{n+1})
\]
for all $n \ge 0$. Since $\set{1-\tfrac{1}{n+1}}{n\ge 0}$ is a set of uniqueness for $\htwo$, it follows that
\[
U\chi_{[0,s]}(z) = e^{-w} E_w (z)
\]
for all $z\in \d$, which implies \eqref{50}.
\end{proof}
\subsection{A Formula for the 
Sarason Transform}

\begin{prop}
If $f \in \ltwo$, then
\[
Uf (z) = \frac{1}{1-z} \int_0^1 f(x) x^{\frac{z}{1-z}}dx
\]
for all $z \in \d$.
\end{prop}
\begin{proof}
\begin{align*}
Uf (z) &= \ip{Uf}{k_z}_{\htwo}\\
&=\ip{f}{U^* k_z}_{\ltwo}\\
(Lemma \eqref{sa.lem.10})\qquad &=
\ip{f}{\frac{1}{1-\bar z} x^{\tfrac{\bar z}{1-\bar z}}}_{\ltwo}\\
&=\frac{1}{1-z} \int_0^1 f(x) x^{\frac{z}{1-z}}dx
\end{align*}
\end{proof}
To obtain a nonrigorous, but highly interesting proof of the proposition, 
let us 
 define the Sarason Transform 
  of  a measure $\mu$ on $[0,1]$, $\mathcal{S} \{\mu\}$, to be the holomorphic function
\[
\mathcal{S} \{\mu\} (z) = \frac{1}{1-z} \int x^{\frac{z}{1-z}}d\mu (x),\qquad |z|< 1.
\]
Note that \[
\mathcal{S}(\chi_{[0,s]}) = \sqrt s\  e^{  \ln \sqrt{s}\  \frac{1+z}{1-z}} = U (\chi_{[0,s]}) .
\]
We have
\[
\mathcal{S} \{ \delta_s \} \= \frac{1}{1-z} s^{\frac{1}{1-z}-1}=\frac{1}{1-z} s^{\frac{z}{1-z}}.
\]
So formally we get
\[
Uf = \mathcal{S}  (\int_0^1 f(x)\delta_x\ dx)=\int_0^1 f(x)  \mathcal{S} (\delta_x)\ dx=\int_0^1 f(x) \frac{1}{1-z} x^{\frac{z}{1-z}}dx,
\]
the formula in Proposition 2.17.


\subsection{Transforms of $H, V$ and $X$}

The Hardy operator $H$ was defined by \eqref{eqin2}. Let $X$ denote multiplication by $x$ on $\ltwo$, and
define the 
 Volterra operator $V: \ltwo \to \ltwo$  by $V = X H$, so
\[
V f(x) \= \int_0^x f(t) dt .
\]
All three of these are monomial operators, and have simple descriptions in terms of monomials.
\begin{eqnarray*}
H : x^n &\ \mapsto\ & \frac{1}{n+1} x^n \\
X: x^n &\mapsto & x^{n+1} \\
V: x^n &\mapsto & \frac{1}{n+1} x^{n+1} .
\end{eqnarray*}
If $T$ is a bounded operator on $\ltwo$, let us write $\widehat{T} =  U  T U^*$ for the
unitarily equivalent operator on $\htwo$. Monomial operators then become operators that map kernel 
functions to multiples of other kernel functions, which are adjoints of  weighted composition operators.
For more about weighted composition operators, see 
e.g. \cite{fo64, coga06, krmo07, coko10, bona10, acks18, cp21}.

We shall let $M_g$ denote the operator of multiplication by $g$, and $C_\beta$ denote composition with 
$\beta$. It was observed in \cite{guna08} that 
it is possible
for the product $M_g C_\beta$ to be bounded even when $M_g$ is not.
The following theorem is proved in \cite{amMO}.
\begin{thm}
\label{thmbmo}
The operator $T: \ltwo \to \ltwo$ is a monomial operator if and only if $ \widehat{T^*}: \htwo \to \htwo$ is
a bounded operator 
of the form $M_g C_\beta $ for some holomorphic $\beta: \D \to \D$ and some $g \in \htwo$.
\end{thm}
Let $\gamma(z) = \frac{1}{2-z}$. This maps $\D$ to $\D$, and maps $\frac{n}{n+1}$ to $\frac{n+1}{n+2}$.
Using Lemma \ref{sa.lem.10} and the preceding formulas, it is  easy to verify the following. We shall let $S$ denote the unilateral shift, the operator of multiplication by $z$ on $\htwo$.
\begin{prop}
\label{prophat}
We have
\begin{eqnarray*}
\widehat{H} &\=& 1 - S^* \\
\widehat{X} &\=& S^* C_\gamma^* \\
\widehat{V} &=& (1 - S^*)C_\gamma^*.
\end{eqnarray*}
\end{prop}

The fact that $1-H$ is unitarily equivalent to a backward shift operator was first proved in \cite{bhs65},
and a proof similar to ours is in \cite{kmp06}.

\subsection{The Sarason Transform and ${\rm Lat}(V)$}

The invariant subspaces of the Volterra operator were described by Brodskii \cite{br57} and Donoghue  \cite{don57}.
\begin{thm}
The space $\m \subseteq \ltwo$ is a closed invariant subspace for $V$ if and only if $\m = \ltworho$ for
some $\rho \in [0,1]$.
\end{thm}
 How do these spaces transform under the Sarason Transform?

For $s\in (0,1]$ let $\Phi_s$ be the singular inner function defined by
\[
\Phi_s (z) = e^{\tfrac12 \ln s \frac{1+z}{1-z}},\qquad z \in \d.
\]
For $s\in [0,1]$, define orthogonal projections $P^\pm_s$ on $\ltwo$ by the formulas
\[
P^-_s f= \chi_{[0,s]} f\ \ \text{ and }\ \ P^+_s f= \chi_{[s,1]} f,\qquad f \in \ltwo.
\]
\begin{lem}
\label{lemsin}
\[
U \ran P_s^- = \Phi_s \htwo \text{ and } U \ran P^+_s = \Phi_s {\htwo}^{\perp}
\]
\end{lem}
\begin{proof}
As $\ran P_s^-$ is invariant for $H^*$, it follows from Proposition \ref{prophat} that
$U \ran P_s^-$ is invariant for the shift $S$, and is therefore of the form $u \htwo$ for some inner function $u$
by Berurling's theorem \cite{beu}. Moreover, $u$ is a constant multiple of the projection of $1$ onto the invariant subspace.  By Lemma \ref{lemchi}, the projection of $1$  is $\sqrt{s} \Phi_s$, so $u = \Phi_s$.
\end{proof}

\subsection{The Sarason Transform and the Laplace Transform}
Recall that the Laplace Transform is defined by the formula
\[
\lap{f}(s)=\int_0^\infty e^{-st}f(t)dt.
\]
Further, if for $f\in\ltwo$ we define $f^\sim$ by the formula
\[
f^\sim(t)=e^{-\frac{t}{2}}f(e^{-t}),\qquad t\in(o,\infty),
\]
then the assignment $f\mapsto f^\sim$ is a Hilbert space isomorphism from $\ltwo$ onto $\ltwoinf$. By making the substitution $x=e^{-t}$ we find that
\begin{align*}
\int_0^1 f(x) x^{\frac{z}{1-z}}dx
&=\int_{\infty}^0 f(e^{-t})\ e^{-t\frac{z}{1-z}}\ (-e^{-t})\ dt\\
&=\int^{\infty}_0 e^{-\frac{t}{2}}f(e^{-t})\ e^{-t(\frac{z}{1-z}+\frac12)}\ dt\\
&=\int^{\infty}_0 f^\sim(t)\ e^{-t( \frac12\frac{1+z}{1-z})}\ dt\\
&=\lap{f^\sim}(\frac12\frac{1+z}{1-z}).
\end{align*}
Hence,
\[
Uf(z) = \frac{1}{1-z}\lap{f^\sim}(\frac12\frac{1+z}{1-z}).
\]

 So after changes of variable from $\ltwo$ of the disc to $\ltwoinf$ and from $\htwo$ to $\htwo$ of the right half plane, the Sarason Transform is simply the Laplace Transform.
 
\section{The Inverse Sarason Transform}

We know from Proposition \ref{prophat} that $1 - H^*$ is unitarily equivalent to the unilateral shift.
Let us find what the orthonormal basis $z^n$ in $H^2$ corresponds to. 
This was first done in \cite{kmp06}. In this section we shall use the 
notation $ \phi \sim g$ to mean that the function $f \in \ltwo$ is mapped to $\phi \in H^2$ by the Sarason transform.

\begin{lem}\label{inv.lem.10}
\be\label{inv.10}
(H^*)^j \ 1 = (-1)^j \frac{(\ln x)^j}{j!}
\ee
\end{lem}
\begin{proof}
We proceed by induction. Clearly, \eqref{inv.10} holds when $j=0$. Assume $j \ge 0$ and \eqref{inv.10} holds. Then
\begin{align*}
(H^*)^{j+1} \ 1 &= H^* ((H^*)^j \ 1)\\ 
&= \frac{(-1)^j}{j!} H^*(\ln x)^j\\
&=\frac{(-1)^j}{j!}\int_x^1 \frac{(\ln t)^j}{t} dt\\ 
&=\frac{(-1)^j}{j!}\int_{\ln x}^0 u^j du\\ 
&=(-1)^{j+1} \frac{(\ln x)^{j+1}}{(j+1)!}
\end{align*}
\end{proof}
The following result is proved in \cite{kmp06}; we include a proof for expository reasons.
\begin{lem}\label{inv.lem.20}
\[
z^n \sim \sum_{j=0}^n \binom{n}{j}\frac{(\ln x)^j}{j!}
\]
\end{lem}
\begin{proof} We have $S = 1 - \widehat{H}^*$. Therefore,
\[
z^n=S^n \ 1 \sim (1-H^*)^n\  1.
\]
But using Lemma \ref{inv.lem.10},
\begin{align*}
(1-H^*)^n\  1 &= \sum_{j=0}^n (-1)^j\binom{n}{j}\ (H^*)^j\ 1\\
&=\sum_{j=0}^n (-1)^j\binom{n}{j} \big((-1)^j \frac{(\ln x)^j}{j!}\big)\\ 
&=\sum_{j=0}^n \binom{n}{j}\frac{(\ln x)^j}{j!}.
\end{align*}
\end{proof}
Define 
\[
e_n(x) \ = \ \sum_{j=0}^n \binom{n}{j}\frac{(\ln x)^j}{j!} .
\]
We just proved that $e_n = (1- H^*)^n 1$.
The fact that the functions $e_n$ are an orthonormal basis is already well-known.
Indeed, the Laguerre polynomials
\[
p_n(t) \ = \ \sum_{j=0}^n \binom{n}{j}\frac{(-t)^j}{j!} 
\]
are orthogonal polynomials of norm $1$ in $L^2([0,\infty))$ with weight $e^{-t}$.
By the change of variables $x = \ln \frac 1t$, we get immediately that $e_n(x)$ is an orthonormal basis
for $\ltwo$.

\begin{lem}\label{inv.lem.30}
If $f \in \htwo$ extends to be analytic on a neighborhood of 1, then
\[
f(z) \sim \sum_{j=0}^\infty f^{(j)}(1)\frac{(\ln x)^j}{(j!)^2}
\]
\end{lem}
\begin{proof}
If
\[
f(z) = \sum_{n=0}^\infty a_n z^n
\]
is the power series representation of $f$, then by Lemma \ref{inv.lem.20}
\begin{align*}
f(z) &\sim \sum_{n=0}^\infty a_n \big(\sum_{j=0}^n \binom{n}{j}\frac{(\ln x)^j}{j!}\big)\\ \\
&=\sum_{j=0}^\infty \frac{(\ln x)^j}{(j!)^2} \big(\sum_{n=j}^\infty \frac{n!}{(n-j)!}a_n\big)\\ \\
&=\sum_{j=0}^\infty f^{(j)}(1)\frac{(\ln x)^j}{(j!)^2}.
\end{align*}
\end{proof}
As a reality check let us let us verify the formula in Lemma \ref{sa.lem.10} using Lemma \ref{inv.lem.30}. Note that
\[
k_{\alpha}^{(j)}(1) =\frac{j!\  \bar\alpha^j}{(1-\bar\alpha)^{j+1}}
\]
Therefore, Lemma \ref{inv.lem.30} implies that
\begin{align*}
k_\alpha (z) &\sim \sum_{j=0}^\infty \frac{j!\  \bar\alpha^j}{(1-\bar\alpha)^{j+1}} \frac{(\ln x)^j}{(j!)^2}\\ 
&=\frac{1}{1-\bar\alpha}\ \sum_{j=0}^\infty\ \frac{1}{j!}\  \big(\frac{\bar\alpha}{1-\bar\alpha} \ln x\big)^j\\ 
&=\frac{1}{1-\bar\alpha}\  e^{\frac{\bar\alpha}{1-\bar\alpha} \ln x}\\ 
&=\frac{1}{1-\bar\alpha}\  x^{\frac{\bar\alpha}{1-\bar\alpha}}.
\end{align*}
The formula in Lemma \ref{inv.lem.30} reminds one of Bessel functions. Indeed,
\begin{align*}
e^{z-1} &\sim \sum_{j=0}^\infty \frac{(\ln x)^j}{(j!)^2}\\ 
&=\sum_{j=0}^\infty \frac{(-1)^j}{(j!)^2}(-\ln x)^j\\ 
&=\sum_{j=0}^\infty \frac{(-1)^j}{(j!)^2}\big(\frac{\sqrt{\ln x^{-4}}}{2}\big)^{2j}\\ 
&=J_0\big({\sqrt{\ln x^{-4}}}\big)
\end{align*}
\begin{prop}\label{inv.prop.10}
If $f \in \htwo$ extends to be analytic on a neighborhood of 1,
then $\Phi:(0,1] \to \c$ defined by
\[
\Phi(x) = U^*f(x),\qquad x\in (0,1],
\]
extends holomorphically to $\c \setminus (-\infty,0]$. Furthermore, if $f_{-1}$ denotes the function defined by $f_{-1}(z)=f(-z)$, then
\[
U^*f_{-1}(x)= \Phi(\frac{1}{x}),\qquad x \in (0,1].
\]
\end{prop}
\begin{proof}
Observe that as $f$ is assumed to be analytic on a neighborhood of 1, the Cauchy-Hadamard radius of convergence formula implies that $F$, defined by the formula,
\[
F(w) = \sum_{k=0}^\infty f^{(k)}(1)\frac{w^k}{(k!)^2},
\]
is an entire function. Consequently, as Lemma \ref{inv.lem.30} implies that
\[
\Phi(x)=U^*f(x) = F(\ln x)
\]
and $\ln x$ extends holomorphically to $\c \setminus (-\infty,0]$, so also, $\Phi$ extends holomorphically to $\c \setminus (-\infty,0]$.

To see the second assertion of the lemma, note using Lemma \ref{inv.lem.30} (with $f$ replaced with $f_{-1}$),
\begin{align*}
U^*f_{-1}(x)&= \sum_{k=0}^\infty f_{-1}^{(k)}(1)\frac{(\ln x)^k}{(k!)^2}\\ \\
&=\sum_{k=0}^\infty f^{(k)}(1)\frac{(\ln x)^k}{(k!)^2}
\end{align*}
\end{proof}
\section{Proofs and auxiliary results}
\label{secpfs}

\subsection{Convergence of Subspaces}
\label{subsecconsub}

\begin{prop}
\label{proppf1}
Let $\m_n$ be a sequence of closed subspaces of a Hilbert space $\h$.
The following are equivalent.

(i)  $\m_n \to \m$

(ii) $P_{\m_n} P_\m \to P_\m$ in the strong operator topology on $\b(\h)$, and there is no 
larger space $\n \supsetneq \m$ such that $P_{\m_n} P_\n \to P_\n$.

(iii)   $P_{\m_n} P_\m \to P_\m$ in the weak operator topology on $\b(\h)$, and  
larger space $\n \supsetneq \m$ such that $P_{\m_n} P_\n \to P_\n$.
\end{prop}
\bp
(i) $\Rightarrow$ (ii). Let $f \in \h$. Let $P_\m f = g$.
By (i), there exist $g_n \in \m_n$ so that $\| g_n - g \| \to 0$.
Therefore 
\[
\| P_{\m_n} g - P_\m g \| \ \leq \ \| g_n - g \| \to 0 ,
\]
so $P_{\m_n} P_\m \to P_\m$ SOT.

If a space $\n \supseteq \m$ existed for which $P_{\m_n} P_\n \to P_\n$, 
let $h \in \n \ominus \m$. Let $h_n = P_{\m_n} h$.
Then $h_n \in \m_n$ and $h_n \to h$. By (i), this means $h \in \m$, so $h = 0$ and
$\n = \m$.

(ii) $\Rightarrow$ (i).
If $f \in \m$, then $P_{\m_n} f \to f$, so $f \in \lim \m_n$.
If there were some $h = \lim g_n$ for a sequence $g_n \in \m_n$, then
$\n = \m + \c h$ would satisfy $P_{\m_n} P_\n \to P_\n$ SOT. So by (ii), this means $h \in \m$,
so $\m = \lim \m_n$.

(ii) $\Leftrightarrow$ (iii): This is because a sequence of projections in a Hilbert space converge WOT if and only if they converge SOT.
Indeed, suppose $Q_n \to Q$ WOT, and $Q$ and each $Q_n$ is a projection.
Then
\[
\| (Q - Q_n) f \|^2 \= \la Q f , f \ra - 2 \Re \la Q_n f, Qf \ra + \la Q_n f, f \ra \ \to 0 .
\]
\ep

\subsection{
Proof of Theorem \ref{thmbasis}}

\bp We have already shown that $e_n$ are an orthonormal basis. Clearly $J$ is unitary, so must be given by Theorem \ref{thmun} for some $c(s)$ and $\tau(s)$.
As $J 1 = 1$, we have $\tau(0) = 0$ and $c(0) = 1$.
As $J^2 = 1$, we have
$\tau(\tau (s)) = s$ and $c(\tau (s)) c(s) = 1$.

So $\tau$ is an automorphism of $\s$ that fixes $0$ and is period $2$.
Once we know $\tau'(0)$, this will uniquely determine $\tau$.
To calculate $\tau'(0)$, note that
\[
e_1(x) \= 1 +  \frac{\partial}{\partial s} x^s |_{s=0}.
\]
Therefore 
\beq
J e_1 (x) &\= & -1 + \frac{\partial}{\partial s} \left[ c(s) x^{\tau(s)} \right]_{s=0} \\
&=& 1 +  c^\prime(0) + c(0) \tau^\prime (0) \ln x .
\eeq
This yields $\tau'(0) = -1$, so \eqref{eqin21} hold.
\ep

\subsection{Proof of Corollary \ref{cormo}}
\label{subseccormo}

\begin{lemma}
\label{leman}
Suppose $T$ is a bounded monomial operator given by
\be
\label{eqdc1}
 T x^s \= c(s) x^{\tau(s)} .
 \ee
Then $\tau$ is a holomorphic function from $\s$ to $\s$, and $c$ is a holomorphic function on $\s$.
\end{lemma}
\bp
The map $s \mapsto x^s$ is a holomorphic map from $\s$ to $\ltwo$. Therefore, for each $t \in \s$,
the map
\[
s \ \mapsto \ \la x^s, T^* x^t \ra \ = \ \frac{c(s)}{1 + \tau(s) + \bar t} 
\]
is holomorphic. Letting $t=0$ and $1$ and taking the quotient, we get
\[
\frac{2 + \tau(s) }{1 + \tau (s) } \= 1 + \frac{1}{1 + \tau(s)}
\]
is a meromorphic function of $s$. Hence $\tau$ is meromorphic in $\s$.
Moreover, $\tau$ cannot have a pole, since otherwise in a neighborhood of this pole it would
take on all values in a neighborhood of $\infty$, including ones not in $\s$.
Therefore $\tau$ is holomorphic, and consequently so is $c(s)$ since we have
\[
c(s) \= (1 + \tau(s) ) \la x^s, T^* 1 \ra .
\]
\ep

\bp (Of Corollary \ref{cormo}).
If $T$ maps ${\rm Lat(H)}$ to ${\rm Lat}(H)$, it must be a monomial operator, since each monomial functions spans a one-dimensional $H$-invariant subspace.

Conversely, suppose $T$ is a monomial operator given by \eqref{eqdc1}, and $\m \in {\rm Lat }(H)$.
 By Lemma \ref{leman}, the function $\tau$ is a holomorphic map from $\s$ to $\s$.
 Define the function $\phi \in H^\infty(\D (1,1))$ by 
 \[
 \phi(z) \= \frac{1}{1 + \tau( \frac{1-z}{z}) } .
 \]
 Then for every $s \in \s$ we have
 \[
 \phi(\frac{1}{1+s}) \= \frac{1}{1 + \tau(s)}.
 \]
 Therefore $HT = T \phi(H)$, since they agree on all monomials, 
 so $H T \m = T \phi(H) \m \subseteq T \m$, as required.
\ep

\begin{rem}
 We used the fact that if $\phi \in H^\infty(\D (1,1))$ then $\phi(H)$ is a bounded operator.
We define $\phi(H)$ to be the monomial operator
\[
\phi(H) : x^s \ \mapsto \  \phi(\frac{1}{1+s})  x^s .
\]
This will be bounded by $M$ if and only if $M^2 - \phi(H)^* \phi(H) \geq 0$, which is equivalent to
\be
\label{eqdc2}
\frac{M^2 - \phi(\frac{1}{1+s})  \overline{\phi(\frac{1}{1+t})} }{1 + s + \bar t} \ \geq \ 0 .
\ee
The fact that \eqref{eqdc2} is equivalent to the assertion that $\phi$ has norm at most $M$
in $H^\infty(\D (1,1))$ is, after a change of variables, the content of Pick's theorem \cite{pi15}.
\end{rem}

\subsection{Proof of Theorem \ref{thmsing1}}

\bp
Let $T$ be a unitary monomial operator, given by 
\be
\label{eqdd1}
 T x^s \= c(s) x^{\tau(s)} .
 \ee
  By Theorem \ref{thmun}, we know that $\tau$ is a holomorphic automorphism of $\s$.
It is well-known that holomorphic automorphisms of the upper half plane are given by linear fractional 
transformations with coefficients from $SL(2,\r)$. So any holomorphic automorphism of $\s$ is of the form
\be
\label{eqaut}
\tau(s) \= \frac{ A (s + \frac 12) - iB}{iC(s + \frac12) + D }- \frac{1}{2} ,
\ee
where
$ \begin{pmatrix} A&B\\C&D \end{pmatrix}$ is in $SL(2,\r)$.
Let $\sigma(s) = \frac{-s}{1+2s}$.
 
 Case (i): $\tau(\i) = \i$. Then $\tau$ is of the form
 $\tau(s) = \alpha s + \beta + i \gamma$, where $\alpha > 0$, $\beta, \gamma \in \r$, and
 $\beta = \frac{\alpha -1}{2}$.
 As $T$ is given by \eqref{eqdd1} and
 \[
 c(s) \= c_0 \frac{1 + \bar \beta + \alpha s + \beta}{1+s} \= c_0 \alpha ,
 \]
 we have
 \[
 T: x^s \ \mapsto \ c_0 \alpha  x^\beta x^{\alpha s}.
 \]
 Therefore
 \[
 T : f(x) \ \mapsto \ c_0 \alpha x^\beta f(x^\alpha) .
 \]
 Therefore
 \[
 T \ L^2( [\rho,1]) \= L^2 [\rho^{\frac{1}{\alpha}}, 1] \= \A_{1, \frac{1}{2\alpha} \log \frac{1}{\rho}} .
 \]
 \vskip 5pt
 Case (ii): $\tau(\i) = - \frac{1}{2}$. Then $\sigma \circ \tau (\i) = \i$.
 So by Case (i), we have
 \[
 J T \ L^2( [\rho,1]) \= A_{1,w}
 \]
 for some $w$. Therefore
 \[
 T \ L^2( [\rho,1]) \= A_{-1,w}.
 \]
 
 \vskip 5pt
 Case (iii) $\tau(\i) = - \frac{1}{2} + i \delta$.
 Then \[
 M_{x^{-i\delta}} T : x^s \mapsto \tilde{c}(s) x^{\tilde \tau (s)} ,
 \]
 where $\tilde \tau(\i) = - \frac12$.
 By Case (ii), we have
 \[
 T \ L^2( [\rho,1]) \=  M_{x^{i\delta}}  \A_{-1,\wp} \= \A_{\tau,w} ,
 \]
 for $\tau = \frac{2i\delta  + 1}{2i \delta - 1}$ and $w = \frac{\wp}{1 + 4\delta^2}$.
 
  \vskip 5pt
 Uniqueness of representation: We need to show that if 
 $\A_{\tau,w} = \A_{\tau',w'}$, then $\tau' = \tau$ and $w' = w$.
 Observe that if $U$ is a unitary, and $P_\m f = g$, then $P_{U\m} Uf = U g$.
 
 Let us calculate $P_{\A_{\tau,w}} x^s$.
 If $\tau = 1$, then
 \[
 P_{\A_{1,w}} x^s \= \chi_{[e^{-2w}, 1]} x^s , 
 \]
 and
 \be
 \label{eqdd2}
 \|  P_{\A_{1,w}} x^s \|^2 \= \frac{1}{1 + 2 \Re s} \left[ 1 - e^{-2w(1 + 2 \Re s)} \right].
 \ee
 Otherwise, $\tau = \frac{2i\delta +1}{2i\delta - 1}$ for some $\delta \in \r$.
 Then 
 \[
  P_{\A_{\tau,w} } x^{s} \=  M_{x^{i\delta}} J P_{\A_{1,\wp}} J x^{s - i \delta} .
  \]
  From \eqref{eqin21},
  \[
  J x^{s - i \delta} \= \frac{1}{ 1 +2s - 2i \delta } x^{ \frac{-s + i \delta}{1 + 2s - 2i \delta}} .
  \]
  Therefore
  \beq
  \|   P_{\A_{\tau,w} } x^{s}\|^2 &\=& \|  P_{\A_{1,\wp}} J x^{s - i \delta}  \|^2 \\
  &=&  \frac{1}{ |1 +2s - 2i \delta|^2 } \int_{e^{-2\wp}}^1  x^{ 2 \Re \frac{-s + i \delta}{1 + 2s - 2i \delta}}\  dx
  \eeq
  When $s = u+iv$ , this gives
  \be
  \label{eqdd3}
    \|   P_{\A_{\tau,w} } x^{s}\|^2 \=
    \frac{1}{1+2u} \left[ 1 - e^{-2\wp \frac{1+2u}{(1+2u)^2 + 4 (\delta - v)^2}} \right].
    \ee
Comparing \eqref{eqdd2} and \eqref{eqdd3}, we see that $\tau$ and $w$ are completely determined
by
$    \|   P_{\A_{\tau,w} } x^{s}\|^2 $.
 \ep

\section{Proofs using Hardy space theory}
\label{sechar}

Although Theorems \ref{thmdecomp} and \ref{thmsing2} are stated without using the language of Hardy spaces, the authors do not know how to prove them directly.

\subsection{Proof of \ref{thmdecomp}}

\bp
Let $\m$ be in ${\rm Lat}(H)$. Define a sequence $S \subset \s$ by
$S = \{ s : \la f, x^s \ra = 0 \ \forall\  f \in \m \}$.
The number $s$ will occur in $S$ with multiplicity $m$ where $m$ is the largest number so that
$\m \perp \{ x^s, (\ln x) x^s, \dots, (\ln x)^{m-1} x^s \}$.
Let $\m_0 = \m(S)$.

To see that $\m = \overline{ \m_0 + \m_1}$ for some singular space $\m_1$, we use the Sarason transform to move to $H^2$.
Then $\m$ becomes $(BS H^2)^\perp$, where $B$ is a Blaschke product and $S$ is a singular inner function. 
As $BS H^2 = B H^2 \cap S H^2$, we have
\[
(BS H^2)^\perp \= \overline{ (B H^2)^\perp + (S H^2)^\perp },
\]
and $\m_1$ is the inverse Sarason transform of $(S H^2)^\perp$.
\ep

\subsection{Proof of \ref{thmsing2}}

Let $S_{\tau, w}$ denote the singular inner function
\[
S_{\tau, w} (z) \= \exp \left( - w \frac{\tau +z}{\tau-z} \right) .
\]
 Let $U : L^2 \to H^2$ be the Sarason transform. By Lemma \ref{lemsin} we have that
\[
 U \A_{1, w} U^* = \left( S_{1, w} H^2 \right)^\perp.
\]
We wish to extend this to other values of $\tau$.
\begin{lemma}
\label{lemsin1}
\be
\label{eqsin2}
 U M_{x^{ic}} U^* (S_{-1,\wp}) \= F S_{\tau, w} ,
 \ee
 where $\tau = \frac{2ic + 1}{2ic - 1}$, $w = \frac{1}{1+4c^2} \wp$ and $F(z) = \exp (-2icw )\frac{1}{1+ic -ic z} $.
\end{lemma}
\bp
Observe first that $\widehat{M_{x^{-ic}} }$ is a unitary operator that takes $k_\alpha$ to
$\overline{\phi (\alpha)} k_{\psi(\alpha)}$, where
\beq
\phi(z) &\=& \frac{1}{1+ic - ic z} \\
\psi(z) &\=& \frac{(1-ic) z+ ic }{ 1+ic  -icz} .
\eeq

Therefore
\[
\widehat{M_{x^{-ic}} } \= C_\psi^* M_\phi^* ,
\]
and so
\be
\label{eqsin4}
\widehat{M_{x^{ic}} } \=  M_\phi C_\psi.
\ee
We have
\[
C_\psi S_{-1,\wp} (z) \= \exp \left( \wp\ \frac{-1 + \psi(z)}{-1 - \psi(z)} \right).
\]

A calculation shows that
\[
 \frac{-1 + \psi(z)}{-1 - \psi(z)}
\=
\frac{1}{1+4c^2} \ \frac{\tau + z}{\tau - z} - \frac{2ic}{1+4c^2} .
\]
Therefore $C_\psi S_{-1,\wp} (z)$ is a unimodular constant times $S_{\tau,w}$, and  \eqref{eqsin2} holds.
\ep
\begin{lemma}
\label{lemsin2}
\be
\label{eqsin3}
 U \A_{\tau, w} U^* = \left( S_{\tau, w} H^2 \right)^\perp.
\ee
\end{lemma}
\bp
We have already proved the case $\tau =1$, so assume $\tau \neq 1$.
Consider next $\tau = -1$.
Then
\beq
U \A_{-1, w} U^* &\=& U J \A_{1,w} U^* \\
 &=& UJ U^* U \A_{1,w} U^* \\
 &=&UJU^* \left( S_{1, w} H^2 \right)^\perp.
\eeq
As $UJU^* f(z) = f(-z)$, we  have
\[
UJU^* \left( S_{1, w} H^2 \right) \=  \left( S_{-1, w} H^2 \right),
\]
so
\[
UJU^* \left( S_{1, w} H^2 \right)^\perp \=  \left( S_{-1, w} H^2 \right)^\perp .
\]
For $\tau \neq \pm 1$, we have, with $\phi$ and $\psi$ as in \eqref{eqsin4} and $F$ as in Lemma \ref{lemsin1},
\beq
 U \A_{\tau, w}^\perp U^* &\=& U M_{x^{ic}} J \A_{1,\wp}^\perp U^* \\
 &=&U M_{x^{ic}} U^* \left( S_{-1, \wp} H^2 \right)\\
 &\=& \{ \phi(z) F(z) S_{\tau,w} (z) h(\psi(z) ) : h \in H^2 \}.
\eeq
As $F$ and $\phi$ are outer and $\psi$ is an automorphism of $\D$, this proves that
\[
U \A_{\tau, w}^\perp U^* \=  S_{\tau, w} H^2 , 
\]
and hence \eqref{eqsin3}.
\ep
\bp (Of Theorem \ref{thmsing2}.)
From Lemma \ref{lemsin2}, we have, for distinct points $\tau_k$,
\beq
U \ \m(\sum_{k=1}^n w_k \delta_{\tau_k})\ U^*
&\=&
\sum_{k=1}^n \left( S_{\tau_k, w_k} H^2 \right)^\perp \\
&=& \left( (\prod_{k=1}^n S_{\tau_k, w_k})\ H^2 \right)^\perp.
\eeq
Suppose that $\mu_n \to \mu$ weak-*, where $\mu$ and each $\mu_n$ are singular.
Define singular inner functions by
\beq
\varphi_n(z) &\=& \exp \left[ - \int \frac{e^{i \theta} + z}{e^{i \theta} - z} d \mu_n (\theta)\right] \\
\varphi(z) &\=& \exp \left[ - \int \frac{e^{i \theta} + z}{e^{i \theta} - z} d \mu (\theta) \right].
\eeq
Then $\| \varphi_n - \varphi \|_{H^2} \to 0$. Indeed, $\varphi_n $ tends to $\varphi$ weakly in $H^2$, since the
functions all have norm $1$ and converge pointwise on $\D$. Therefore
\beq
\| \varphi_n - \varphi \|^2 &\=& 2 - 2 \Re \la \varphi_n , \varphi \ra \\
& \to & 0.
\eeq
This means that not only do the Toeplitz operators $T_{\overline{\varphi}_n}$ converge to $T_{\overline{\varphi}}$ in the strong operator topology, but $T_{{\varphi_n}}T_{\overline{\varphi}_n}$
converges to $T_\varphi T_{\bar \varphi}$ SOT.
This is proved in \cite[p. 34]{nik}; for the convenience of the reader, we include the proof.
Let $f \in H^2$. Then
\beq
\| T_{{\varphi_n}}T_{\bar{\varphi}_n} f - T_{{\varphi}}T_{\overline{\varphi}}f \| & \ \leq \ &
	\| T_{{\varphi_n}}(T_{\overline{\varphi}_n} - T_{\overline{\varphi}} ) f \| + 
\| (T_{{\varphi_n}} - T_\varphi ) T_{\overline{\varphi}} ) f \| \\
&\leq& \sup_n \| \varphi_n \|_{H^\i} \| ( T_{\overline{\varphi}_n} - T_{\overline{\varphi}} ) f \| +
\left( \int |\varphi_n - \varphi|^2 |T_{\bar \varphi} f |^2 \right)^{\frac 12}.
\eeq
The first term tends to zero because $T_{\overline{\varphi}_n}$ tends to $T_{\bar \varphi}$ in the SOT, and the second term tends to $0$ because $\varphi_n $ tends to $\varphi$ in measure and $|\varphi_n - \varphi| \leq 2$.
As $T_{{\varphi_n}}T_{\overline{\varphi}_n}$ is the projection onto $(\varphi_n H^2)^\perp$, this means by
Proposition \ref{proppf1} that the spaces $(\varphi_n H^2)^\perp$ converge to $(\varphi H^2)^\perp$. Applying the inverse Sarason transform, we conclude that $\m(\mu_n)$ converges to $\m(\mu)$.
\ep

\bibliography{../references_uniform_partial}
\end{document}